\documentclass{amsart}
\usepackage{amsmath, amsthm, amssymb, amscd, mathrsfs}
\begin{document}

\newtheorem{theorem}{Theorem}[section]
\newtheorem{lemma}[theorem]{Lemma}
\newtheorem{corollary}[theorem]{Corollary}
\newtheorem{conjecture}[theorem]{Conjecture}
\newtheorem{question}[theorem]{Question}
\newtheorem{problem}[theorem]{Problem}
\newtheorem*{claim}{Claim}
\newtheorem*{criterion}{Criterion}
\newtheorem*{independent_lemma}{Lemma ?.?}
\newtheorem*{quasimorphism_combing_theorem}{Theorem ?.?}
\newtheorem*{central_limit_theorem}{Theorem ?.?}

\theoremstyle{definition}
\newtheorem{definition}[theorem]{Definition}
\newtheorem{construction}[theorem]{Construction}
\newtheorem{notation}[theorem]{Notation}
\newtheorem{convention}[theorem]{Convention}
\newtheorem*{warning}{Warning}

\theoremstyle{remark}
\newtheorem{remark}[theorem]{Remark}
\newtheorem{example}[theorem]{Example}

\def\area{\text{area}}
\def\id{\text{id}}
\def\1{{\bf{1}}}
\def\H{\mathbb H}
\def\Z{\mathbb Z}
\def\R{\mathbb R}
\def\C{\mathbb C}
\def\F{\mathcal F}
\def\CC{{\mathcal C}}
\def\BC{{\mathcal{BC}}}
\def\A{{\mathsf A}}
\def\B{{\mathsf B}}
\def\G{\Gamma}
\def\out{\textnormal{Out}}
\def\Pic{\textnormal{Pic}}
\def\Cut{\textnormal{Cut}}
\def\aut{\textnormal{Aut}}
\def\MCG{\textnormal{MCG}}
\def\til{\tilde}
\def\length{\textnormal{length}}
\def\axis{\textnormal{axis}}
\def\nul{{\varnothing}}
\def\almost{almost semisimple}
\def\cone{\textnormal{cone}}

\title{Combable functions, quasimorphisms, and the central limit theorem}
\author{Danny Calegari}
\address{Department of Mathematics \\ Caltech \\
Pasadena CA, 91125}
\email{dannyc@its.caltech.edu}
\author{Koji Fujiwara}
\address{Graduate School of Information Science \\ Tohoku University \\ Sendai, Japan}
\email{fujiwara@math.is.tohoku.ac.jp}

\date{6/13/2009, Version 0.25}

\begin{abstract}
A function on a discrete group is weakly combable if its discrete derivative with respect
to a combing can be calculated by a finite state automaton. A weakly combable function is
bicombable if it is Lipschitz in both the left and right invariant word metrics. 

Examples of bicombable functions on word-hyperbolic groups include
\begin{enumerate}
\item{homomorphisms to $\Z$}
\item{word length with respect to a finite generating set}
\item{most known explicit constructions of quasimorphisms (e.g. the Epstein-Fujiwara
counting quasimorphisms)}
\end{enumerate}

We show that bicombable functions on word-hyperbolic groups satisfy a {\em central limit theorem}:
If $\overline{\phi}_n$ is the value of $\phi$ on a random element of word length $n$
(in a certain sense), there are $E$ and $\sigma$ for which there is convergence in the sense of distribution
$n^{-1/2}(\overline{\phi}_n - nE) \to N(0,\sigma)$, where $N(0,\sigma)$ denotes the normal
distribution with standard deviation $\sigma$. As a corollary, we show that if $S_1$ and $S_2$ are
any two finite generating sets for $G$, there is an algebraic number $\lambda_{1,2}$ depending on $S_1$ and $S_2$
such that almost every word of length $n$ in the $S_1$
metric has word length $n\cdot \lambda_{1,2}$ in the $S_2$ metric, with error of size
$O(\sqrt{n})$.
\end{abstract}

\maketitle

\section{Introduction}

This paper concerns the statistical distribution of values of certain functions
on hyperbolic groups, and lies at the intersection of computer science, ergodic
theory, and geometry. Since Dehn, or more recently Gromov (e.g. \cite{Gromov_hyperbolic})
it has been standard practice to study finitely generated groups geometrically, via
large-scale properties of their Cayley graphs. Geometric spaces are probed effectively
by Lipschitz functions; in a finitely generated group one has the option of probing the geometry
of the Cayley graph by functions which are Lipschitz with respect to both the left- and 
right-invariant metrics simultaneously. One natural source of such functions are
{\em homomorphisms} to $\R$. For instance, Rivin \cite{Rivin} and
Sharp \cite{Sharp} have studied the abelianization map to $H_1$ on a free group, 
and derived strong results (also, see Picaud \cite{Picaud}). 
Unfortunately, many groups (even hyperbolic groups) admit no nontrivial
homomorphisms to $\R$ or $\Z$; however, in many cases such groups admit a rich (uncountable
dimensional!) family of {\em quasimorphisms} --- informally, homomorphisms up to bounded error.
Quasimorphisms arise in the theory of bounded cohomology, and in the context of a number
of disparate extremal problems in topology; see e.g. \cite{Calegari_scl_monograph} for an introduction.
The defining property of a quasimorphism implies that it is Lipschitz in both the left- and
right-invariant metrics on a group, and therefore one expects to learn a lot about the
geometry of a group by studying its quasimorphisms.

A class of groups for which this study is particularly fruitful is that of
{\em word-hyperbolic groups}. As Cannon \cite{Cannon_automatic} and Gromov \cite{Gromov_hyperbolic} 
pointed out, and Coornaert and Coornaert-Papadopolous
elaborated (\cite{Coornaert}, \cite{Coornaert_Papadopoulos_book}, \cite{Coornaert_Papadopoulos_coding}),
a hyperbolic group can be parameterized (i.e. {\em combed}\/) by walks on a finite directed graph, and
the (symbolic) dynamics of the (combinatorial) geodesic flow 
can be recast in terms of the theory of subshifts of finite type. ``Most'' groups (in a suitable sense)
are word-hyperbolic,
so this gives rise to a rich family of natural examples of dynamical systems with hyperbolic dynamics
(in the sense of ergodic theory). Underlying this machinery is
the fact that geodesics in a hyperbolic group can be recognized by a particularly simple class
of computer program, namely the kind which can be implemented with a {\em finite state automaton}.

Putting these two things together, one is naturally led to study {\em weakly combable functions} ---
informally, those whose (discrete) derivative can be calculated by a finite state automaton ---
and to specialize the study to the class of {\em bicombable} functions, namely those which are
Lipschitz in both the left- and right-invariant metrics.

Many naturally occurring classes of functions on word hyperbolic groups are bicombable,
including
\begin{enumerate}
\item{Homomorphisms to $\Z$}
\item{Word length with respect to any finite generating set (not necessarily symmetric)}
\item{Epstein-Fujiwara counting functions (see \cite{Epstein_Fujiwara})}
\end{enumerate}
Bicombability of a function does not depend on a particular choice of combing, or even
a generating set. Moreover, the set of all bicombable functions on a given group is
a free Abelian group (of infinite rank). Although surely interesting in their own
right, such functions are inexorably tied to the theory of quasimorphisms, in at least two ways.
Firstly, the best-studied class of quasimorphisms on hyperbolic groups, the so-called
Epstein-Fujiwara counting functions (i.e. bullet (3) above), turn out to be bicombable.
Secondly, given any bicombable function $\phi$, its antisymmetrization $\psi$ defined by the
formula $\psi(g) = \phi(g) - \phi(g^{-1})$ is a quasimorphism. A particularly simple
construction is as follows. Let $S$ be a finite generating set for $G$ (as a semigroup), 
and $S^{-1}$ the set of elements inverse to $S$ in $G$. Let $\phi_S:G \to \Z$ compute word length
with respect to the generating set $S$, and $\phi_{S^{-1}}:G \to \Z$ compute word length
with respect to the generating set $S^{-1}$. Then both $\phi_S$ and $\phi_{S^{-1}}$ are
bicombable, and their difference $\psi_S:=\phi_S - \phi_{S^{-1}}$ is a bicombable quasimorphism.
Typically $\psi_S$ will be highly nontrivial. 

\medskip

The main result in this paper is a {\em Central Limit Theorem} (Theorem~\ref{clt_bicombable}) 
for bicombable functions on
a word hyperbolic group. If $\phi$ is a bicombable function on $G$, and $\overline{\phi}_n$
denotes the value of $\phi$ on a random element in $G$ of length $n$ (with respect to some
fixed generating set $S$ and with a suitable definition of random to be made precise), there are
real numbers $E$ and $\sigma$ so that there is convergence in the sense of distribution
$$n^{-1/2}(\overline{\phi}_n - nE) \to N(0,\sigma)$$
where $N(0,\sigma)$ denotes the normal (Gaussian) distribution with standard deviation $\sigma$
(the possibility $\sigma = 0$ is allowed, in which case $N(0,\sigma)$ just denotes the
Dirac distribution supported at the origin). If $\phi$ is a quasimorphism, then symmetry
forces $E=0$, so as a corollary one concludes that for any bicombable quasimorphism $\phi$ on a word-hyperbolic
group, and for any
constant $\epsilon > 0$, there is an $N$ and a $K$ such that for all $n \ge N$, there
is a subset $G_n'$ of the set of all words $G_n$ of length $n$ so that $|G_n'|/|G_n| \ge 1-\epsilon$,
and $|\phi(g)| \le K\sqrt{n}$ for all $g \in G_n'$.

Another corollary (Corollary~\ref{genset_length_compare}) 
of independent interest to geometric group theorists, is the following. Let
$S_1$ and $S_2$ be two finite generating sets for a non-elementary word hyperbolic group $G$. Then
there is an algebraic number $\lambda_{1,2}$ such that for any $\epsilon > 0$ there is a constant $K$ 
and an $N$
so that if $G_n$ denotes the set of elements of length $n \ge N$ in the $S_1$ metric, there is a subset
$G_n'$ with $|G_n'|/|G_n| \ge 1-\epsilon$, so that for all $g \in G_n'$ there is an inequality
$$\big| |g|_{S_1} - \lambda_{1,2} |g|_{S_2} \big| \le K\cdot \sqrt{n}$$
(of course, $|g|_{S_1} = n$ in this formula).

This observation can probably be generalized considerably. For instance, if $H$ is
a quasiconvex subgroup of a hyperbolic group $G$, then word length in $G$ (in some generating
set) is bicombable as a function on $H$. It follows that for any finite generating sets $S_H$ for $H$
and $S_G$ for $G$, there is some constant $\lambda$
such that almost every word of length $n$ in $H$ in the $S_H$ metric has length $\lambda \cdot n$
in $G$ in the $S_G$ metric, with error of order $\sqrt{n}$.

\medskip

We say a word about the method of proof. In general terms, to derive a central limit theorem
for random sums in a dynamical system one needs two ingredients: independence (i.e. hyperbolicity)
and recurrence. Informally, one obtains a Gaussian distribution by adding up {\em independent}
samples from the {\em same} distribution. In our context hyperbolicity, or independence of trials,
comes from the finite size of the graph (and the finite memory of the automaton)
parameterizing elements of the group. As is well-known to the experts, the difficult point is
recurrence: a graph parameterizing a combing of a hyperbolic group may typically fail to be
recurrent. In fact, it may have many recurrent subgraphs whose adjacency matrices have largest
real eigenvalue equal to the largest real eigenvalue for the adjacency matrix of the graph as
a whole. In place of recurrence in this graph, we use recurrence {\em at infinity}, using
Coornaert's theorem that the action of $G$ on $\partial G$ is ergodic with respect to a
Patterson-Sullivan measure.

\medskip

Statistical theorems for values of (certain) quasimorphisms on certain groups have been
derived by other authors. Horsham-Sharp (\cite{Horsham_Sharp}) recently showed that 
a {\em H\"older continuous} quasimorphism (see \S~\ref{Holder_section}) 
on a free group $F$ satisfies a central limit theorem. In quite a different direction,
Sarnak \cite{Sarnak_private} showed that the values of the {\em Rademacher function} on $\text{PSL}(2,\Z)$,
sorted by geodesic length, form a Cauchy distribution. Here the Rademacher function
is a correction term to the logarithm of the Dedekind eta function, and arises in elementary
number theory. Thinking of $\text{PSL}(2,\Z)$ as a subgroup of the group of isometries of the hyperbolic
plane, every element whose trace has absolute value greater than $2$ is represented by an
isometry with an axis, with a well-defined (geodesic) translation length. The reason for the
appearance of the Cauchy distribution instead of a Gaussian one has to do with the
nonlinear relation between geodesic length and word length in $\text{PSL}(2,\Z)$, and the
noncompactness of the quotient $\H^2/\text{PSL}(2,\Z)$. Of course, the Rademacher function is
bicombable, and therefore our main theorem shows that its values have a Gaussian distribution
with respect to any word metric on the word-hyperbolic group $\text{PSL}(2,\Z)$.

\medskip

The plan of the paper is as follows. In \S~2 we introduce basic definitions and properties
of word hyperbolic groups and finite directed graphs (which we call digraphs, following the
computer scientists' convention). In \S~3 we define combable functions and establish some of
their basic properties, proving that the Epstein-Fujiwara counting quasimorphisms on hyperbolic
groups are bicombable. In \S~4 we prove the central limit theorem and derive the main corollaries
of interest. Finally in \S~5 we briefly discuss the results of Horsham-Sharp.

\section{Groups and digraphs}\label{groups_automata_section}

\subsection{Hyperbolic groups}

\begin{definition}
A path metric space is {\em $\delta$-hyperbolic} for some $\delta \ge 0$ if
for every geodesic triangle $abc$ every point in the edge $ab$ is contained
in the union of the $\delta$-neighborhoods of the other two edges:
$$ab \subset N_\delta(bc) \cup N_\delta(ca)$$
A group $G$ with a finite generating set $S$ is {\em $\delta$-hyperbolic}
if the Cayley graph $C_S(G)$ is $\delta$-hyperbolic as a path metric space.

$G$ is {\em word-hyperbolic} (or simply {\em hyperbolic})
if there is a $\delta \ge 0$ and a finite generating
set $S$ for which $C_S(G)$ is $\delta$-hyperbolic.
\end{definition}

A word-hyperbolic group is {\em elementary} if it contains a finite
index cyclic subgroup. If $G$ is nonelementary, it contains many quasi-isometrically
embedded nonabelian free groups of arbitrary rank, by the ping-pong lemma.

\medskip

We assume the reader is familiar with basic elements of coarse geometry:
$(k,\epsilon)$-quasi-isometries, quasigeodesics, etc. 
We summarize some of the main properties of $\delta$-hyperbolic spaces below:

\begin{theorem}[Gromov; basic properties of $\delta$-hyperbolic spaces]\label{delta_hyperbolic_space_facts}
Let $X$ be a $\delta$-hyperbolic path metric space.
\begin{enumerate}
\item{{\bf Morse Lemma}. For every $k,\epsilon$
there is a universal constant $C(\delta,k,\epsilon)$ such that every $(k,\epsilon)$-quasigeodesic
segment with endpoints $p,q \in X$ lies in the $C$-neigh\-bor\-hood of any geodesic
joining $p$ to $q$.}\label{morse_bullet}
\item{{\bf Quasigeodesity is local}. For
every $k,\epsilon$ there is a universal constant $C(\delta,k,\epsilon)$ such that
every map $\phi:\R \to X$ which restricts on each segment of length $C$ to a
$(k,\epsilon)$-quasigeodesic is (globally) $(2k,2\epsilon)$-quasigeodesic.}
\item{{\bf Ideal boundary}. There is an {\em ideal boundary} $\partial X$ functorially
associated to $X$, whose points consist of quasigeodesic rays up to the equivalence
relation of being a finite Hausdorff distance apart, which is metrizable. If $X$
is proper, $\partial X$ is compact, and any quasi-isometric embedding $X \to Y$
between hyperbolic spaces induces a continuous map $\partial X \to \partial Y$.}
\end{enumerate}
\end{theorem}

For proofs and a more substantial discussion, see \cite{Gromov_hyperbolic}.

\subsection{Regular languages}

We give a brief overview of the theory of regular languages and automata. There are many
conflicting but equivalent definitions in the literature; we use a parsimonious
formulation which is adequate for this paper.
For a substantial introduction to the theory of regular languages see \cite{ECHLPT}.

\begin{definition}
An {\em alphabet} is a finite set whose elements are called {\em letters}. 
A {\em word} (or sometimes {\em string}\/)
on an alphabet is a finite ordered sequence of letters. A {\em language} is a subset
of the set of all words on a fixed alphabet. A language is {\em prefix closed} if
every prefix of an element of the language is also in the language.
\end{definition}

If $S$ is an alphabet, the set of all words on $S$ is denoted $S^*$, and a language is
a subset of $S^*$. 

\begin{convention}\label{word_convention}
Throughout this paper, if $w$ denotes a word in $S^*$, then $w_i$ will denote the prefix of length $i$
and $w(i)$ will denote the $i$th letter. The length of $w$ is denoted $|w|$.
Hence (for example), $w_0$ denotes the empty word, and $|w_i|=i$.
\end{convention}

\begin{definition}
A {\em finite state automaton} (or just automaton for short) on a fixed alphabet $S$
is a finite directed graph $\Gamma$ with a distinguished initial vertex $v_1$
(the {\em input state}\/), and a choice of a distinguished subset $Y$ of the vertices
of $\Gamma$ (the {\em accept states}\/), whose oriented edges are labeled by letters of $S$
in such a way that at each vertex there is at most one outgoing edge with any given label.
\end{definition}
 
The vertices are called the {\em states} of the automaton. A
word $w$ in the alphabet determines a directed path $\gamma$ in an automaton,
defined as follows. The path $\gamma$ starts at the initial vertex at time $0$,
and proceeds by reading the letters of $w$ one at a time, and moving along the
directed edge with the corresponding label if one exists, and halts if not.
If the automaton reads to the end of $w$ without halting, the last vertex of $\gamma$
is the {\em final state}, and the word is {\em accepted} if the final state of
$\gamma$ is an accept state, and {\em rejected} otherwise. In this paper, by convention,
if an automaton halts before reaching the end of a word, the word is {\em rejected}.

\begin{definition}
A language $L \subset S^*$ is {\em regular} if it is exactly the set of words accepted
by some finite state automaton. 
\end{definition}

A regular language $L$ is prefix-closed if and only if it is accepted by some finite state
automaton for which every vertex is an accept state. In this paper, we are only concerned 
with prefix-closed regular languages. Hence a word $w$ in $S^*$ is
accepted if and only if it corresponds to a directed path of length $|w|$.

\begin{convention}\label{path_convention}
Throughout this paper, if
$\gamma$ denotes a directed path in a graph, then $\gamma_i$ will denote the initial subpath of
length $i$, and $\gamma(i)$ will denote the $i$th vertex in the path {\em after the initial vertex}. The
length of $\gamma$ is denoted $|\gamma|$. Hence (for example) $\gamma(0)$ denotes the initial vertex.
\end{convention}

If $\Gamma$ is the underlying directed graph of the automaton,
we say that $\Gamma$ {\em parameterizes} $L$. If $L$ is prefix closed and parameterized
by $\Gamma$, there is a natural bijection between elements of $L$ and paths in $\Gamma$
starting at the initial vertex.

\begin{remark}
Some authors by convention add a terminal ``fail'' state to an automaton, and
add directed labeled edges to this fail state so that at every vertex there is
a full set of outgoing directed labeled edges, one for each element of the alphabet. 
In this paper, automata will only be fed accepted
words, so this convention is irrelevant.
\end{remark}

An automaton is like a computer with finite resources running a specific
program. The language in which the program is written
only allows {\em bounded loops} --- i.e. loops of the form ``for $i=1$ to $10$ do'',
where the parameters of the loop are explicitly spelled out --- in order to
certify that the computer will not run out of resources on execution of the program.

Describing an automaton in terms of vertices and edges is like describing a
program in terms of machine code. It is more transparent, and the result is
both more human readable and more easily checkable,
to describe instead (in more prosaic terms) the program that the automaton carries out.

\subsection{Directed graphs and components}

Now and in the sequel we refer to directed graphs by the computer scientists' term ``digraphs".
A {\em pointed digraph} is a digraph with a choice of an initial vertex.

\begin{definition}
Let $\Gamma$ be a digraph. A vertex $v$ is {\em accessible} from a vertex $u$ if there
is a directed path in $\Gamma$ from $u$ to $v$. A {\em component} is a maximal subgraph
such that for any two vertices $u,v$ in the subgraph (not necessarily distinct)
the state $v$ is accessible from $u$.
\end{definition}

\begin{remark}
A subgraph in which every vertex is accessible from every other is sometimes called
{\em recurrent}, so a component is just a maximal recurrent subgraph. We use both terms
interchangeably in the sequel.
\end{remark}

Different components in a digraph are disjoint. If $\Gamma$ is a pointed digraph,
associated to $\Gamma$ is a digraph $C(\Gamma)$ whose vertices are components of $\Gamma$,
together with an edge from one component to the other if there is a directed path
in $\Gamma$ between the corresponding components which does not pass through any other
component. Note that it is entirely possible that $C(\Gamma)$ is empty.

We can enhance $C(\Gamma)$ to a pointed digraph as follows.
If the initial vertex $v_1$ of $\Gamma$ is contained in a component $C$, the corresponding
vertex $C$ of $C(\Gamma)$ becomes the initial vertex. Otherwise add an initial vertex to
$C(\Gamma)$, and join it by a path to each component $C$ that can be reached in $\Gamma$
by a path from the initial vertex which does not pass through any other component.

\begin{lemma}
There are no oriented loops in $C(\Gamma)$.
\end{lemma}
\begin{proof}
If there is a path from $C$ to itself in $C(\Gamma)$ passing through $C'$ then there
is a loop in $\Gamma$ from some vertex in $C$ to itself passing through $C'$. Hence
every vertex of $C'$ is accessible from every vertex of $C$ and conversely, contrary
to the fact that $C$ and $C'$ are maximal.
\end{proof}

\subsection{Combings}

Let $G$ be a group and $S$ a finite subset which generates $G$ as a semigroup. There is an
evaluation map $S^* \to G$ taking a word in the alphabet $S$ to the corresponding element in
$G$. We denote this map by $w \to \overline{w}$. Note that a word $w \in S^*$ can equally
be thought of as a directed path in the Cayley graph $C_S(G)$ from $\id$ to $\overline{w}$.
Note further if $S$ is not symmetric (i.e. $S \ne S^{-1}$), 
the Cayley graph $C_S(G)$ should really be thought of as a digraph.

\begin{convention}\label{group_convention}
By Convention~\ref{word_convention}, $\overline{w_i}$ denotes the element in $G$ 
corresponding to the prefix $w_i$ of $w$. Hence (for example), $\overline{w_0}=\id$ for any $w$.
By Convention~\ref{path_convention}, if $\gamma$ denotes the corresponding path in $C_S(G)$ from $\id$ to
$\overline{w}$, then (by abuse of notation) we write $\gamma(i)=\overline{w_i}$ as elements of $G$.
\end{convention}

\begin{definition}
Let $G$ be generated as a semigroup by a finite subset $S$. A {\em combing of $G$
with respect to $S$} is a prefix-closed regular language $L \subset S^*$ for which the
evaluation map is a bijection $L \to G$ and such that every element of $L$ is a geodesic.
\end{definition}

Geometrically, let $\Gamma$ be a digraph parameterizing $L$. Let $\til{\Gamma}$ be the
universal cover of $\Gamma$, let $\til{v_1} \in \til{\Gamma}$ be a lift of the initial vertex,
and let $T \subset \til{\Gamma}$ be the subgraph consisting
of all points accessible from $\til{v_1}$. There is a natural
map $T \to C_S(G)$ taking $\til{v_1}$ to $\id$ and directed paths to directed paths. The statement that
$L$ is a combing with respect to $S$ corresponds to the geometric statement that
the image of $T$ is a spanning subtree of $C_S(G)$, and that directed paths in $T$ are taken to
geodesics in $C_S(G)$.

\medskip

Hyperbolic groups admit combings with respect to any finite generating set. In fact,
let $G$ be hyperbolic and let $S$ be a fixed finite generating set. Choose a total order $\prec$
on the elements of $S$, and by abuse of notation,
let $\prec$ denote the induced lexicographic order on $S^*$. That is,
$w \prec w'$ for words $w,w'$ if there are expressions $w = xsy, w' = xs'y'$ for $s,s' \in S$
and $x,y,y' \in S^*$ where $s \prec s'$.

Let $L \subset S^*$ denote the language of {\em lexicographically first geodesics} in $G$.
That is, for each $g \in G$, a word $w \in S^*$ with $\overline{w} = g$ is in $L$ if and only
if $|w| = |g|$, and $w \prec w'$ for all other $w' \in S^*$ with these properties.
It is clear that $L$ is prefix-closed and bijective with $G$.

\begin{theorem}[Cannon \cite{Cannon_automatic}, \cite{ECHLPT}]
Let $G$ be a hyperbolic group, and $S$ a finite generating set. The language of
lexicographically first geodesics is regular (and therefore defines a combing of $G$
with respect to $S$).
\end{theorem}

\begin{remark}
Note that the choice of a digraph $\Gamma$ parameterizing a combing $L$ is {\em not}
part of the data of a combing. A given regular language is typically parameterized
by infinitely many different digraphs.
\end{remark}

\begin{warning}
Many different definitions of combing exist in the literature. They are often, but not always,
assumed to be bijective, and sometimes to be quasigeodesic. See \cite{Rees_hairdressing}
for a survey and some discussion.
\end{warning}

\section{Combable functions}\label{combable_functions_section}

\subsection{Left and right invariant metrics}

Let $G$ be a group with finite generating set $S$ (as a semigroup). There are two
natural metrics on $G$ associated to $S$ --- a {\em left} invariant metric $d_L$
which is just the path metric in the Cayley graph $C_S(G)$, and a {\em right} invariant
metric $d_R$ where $d_R(a,b) = d_L(a^{-1},b^{-1})$. If $|\cdot|$ denotes word length with
respect to $S$, then
$$d_L(a,b) = |a^{-1}b| \text{ and } d_R(a,b) = |ab^{-1}|$$

\begin{remark}
If $S$ is not equal to $S^{-1}$ then this is not strictly speaking a metric,
since it is not symmetric in its arguments. 
\end{remark}

\begin{lemma}\label{left_and_right_lemma}
A function $f:G \to \R$ is Lipschitz for the $d_L$ metric if and only if there is a constant
$C$ so that for all $a \in G$ and $s \in S$ there is an inequality
$$|f(as) - f(a)| \le C$$
Similarly, $f$ is Lipschitz for the $d_R$ metric if
$$|f(sa) - f(a)| \le C$$
\end{lemma}
\begin{proof}
These properties follow immediately from the definitions.
\end{proof}

\begin{remark}
Note the curious intertwining of left and right: a function is Lipschitz in the
{\em left} invariant metric if and only if it changes by a bounded amount under multiplication
on the {\em right} by a generator, and {\it vice versa}.
\end{remark}

\subsection{Combable functions}

Suppose $G$ is a group with finite generating set $S$, and let $L \subset S^*$ be a combing.
Let $\Gamma$ be a digraph parameterizing $L$. There are natural bijections
$$\text{ paths in } \Gamma \text{ starting at the initial vertex }\longleftrightarrow L \longleftrightarrow G$$
If $w \in L$, and $\gamma$ denotes the corresponding path in $\Gamma$, then recall
that Convention~\ref{path_convention} says that
$\gamma(i)$ is the vertex in $\Gamma$ visited after reading the first $i$ letters of $w$.

\begin{definition}\label{combable_function_definition}
Let $G$ be a group with finite symmetric generating set $S$,
and let $L \subset S^*$ be a combing. A function $\phi:G \to \Z$ 
is {\em weakly combable with respect to $S$, $L$} (or {\em weakly combable} for short)
if there is a digraph $\Gamma$ parameterizing $L$, and a function $d\phi$ from the
vertices of $\Gamma$ to $\Z$ such that for any $w \in L$ there is an equality
$$\phi(\overline{w}) = \sum_{i=0}^{|w|} d\phi(\gamma(i))$$
where $\gamma$ is the directed path in $\Gamma$ corresponding to $w$.

The function $\phi$ is {\em combable} if it is weakly combable and Lipschitz in the
$d_L$ metric, and it is {\em bicombable} if it is weakly combable and Lipschitz in both
the $d_L$ and the $d_R$ metrics.
\end{definition}

In words, a function $\phi$ is weakly combable 
if its derivative along elements in a combing can be calculated by an automaton.

\begin{example}\label{word_length_bicombable}
Word length (with respect to a fixed finite generating set $S$, symmetric or not) is
bicombable. In fact, if $L$ is any combing for $S$, and
$\Gamma$ is any digraph parameterizing $L$, the function $d\phi$ taking the value $1$ on
every (noninitial) vertex of $\Gamma$ satisfies the desired properties.
\end{example}

\begin{example}
Let $F$ be the free group with generators $a,b$. 
Let $S = \lbrace a,b,a^{-1},b^{-1} \rbrace$, and let $L$ be the
language of reduced words in the generating set. The function
$$\phi(w) = \begin{cases} 
|w|  & \text{ if } w \text{ begins with } a \\
0 & \text{ otherwise }
\end{cases}$$
is combable but not bicombable.
\end{example}

The definition of weakly combable depends on a generating set and a combing.

\begin{example}\label{big_variation}
Let $G = \Z \times \Z/2\Z$
where $\Z$ is generated by $a$ and $\Z/2\Z$ is generated by $b$. Let
$S = a,a^{-1},b$ and let $L$ be the combing consisting of words of the form
$a^n$ and $ba^n$ for different $n$. Let $S' = b, (ab), (ab)^{-1}$ and let $L'$
be the combing consisting of words of the form $(ab)^n$ and $b(ab)^n$ (each
bracketed term denotes a single ``letter'' of $S'$). Let $\phi:G \to \Z$ be
the function defined by $\phi(a^n) = n$ and $\phi(ba^n)=0$. Then $\phi$ is
weakly combable with respect to $L$, but not with respect to $L'$ since the
derivative $d\phi$ is unbounded along words of $L'$.
\end{example}

However, the property of being (bi-)combable does {\em not} depend on a generating set,
or a choice of combing.

\begin{lemma}\label{combing_independent_lemma}
Let $G$ be a word-hyperbolic group. Suppose $S,S'$ are two finite symmetric
generating sets, and $L,L'$ bijective geodesic combings in $S^*$ and $(S')^*$
respectively. Then the sets of functions which are (bi-)combable with respect to $L$
and with respect to $L'$ are equal.
\end{lemma}
\begin{proof}
The property of being Lipschitz in either the left or right metric is independent of a choice of
generating set (and does not depend on a combing), so it suffices to prove the lemma for combable functions.
Let $\phi$ be combable with respect to $S$, $L$. We will show it is combable with respect to $S'$, $L'$.

Let $w'$ be a word in $L'$, and let $w$ be the word in $L$ with $\overline{w}' = \overline{w}$. Let $\gamma$
be a path in $C_S(G)$ corresponding to $w$, and let $\gamma'$ be a path in $C_{S'}(G)$ corresponding to $w'$.
Although $\gamma$ does not make sense as a {\em path} in $C_{S'}(G)$, nevertheless each vertex $\gamma(i)$
is an element of $G$, and therefore corresponds to a unique vertex in $C_{S'}(G)$. The ``discrete path''
determined by the vertices $\gamma(i)$ is a quasigeodesic in $C_{S'}(G)$, and by the Morse Lemma,
fellow travels the path $\gamma'$. In other words:
\begin{enumerate}
\item{there is a constant $C_1$ so that every $\gamma(i)$ is within distance $C_1$ of some $\gamma'(j)$, and every
$\gamma'(j)$ is within distance $C_1$ of some $\gamma(i)$}
\item{there is a constant $C_2$ so that if $\gamma(i)$ is within distance $C_1$ of $\gamma'(j)$, then
$\gamma(k)$ is within distance $C_1$ of $\gamma'(j+1)$ for some $k$ with $|k-i| \le C_2$}
\end{enumerate}
Here distances are all measured in $C_{S'}(G)$.

\medskip

Suppose we are given some $v' \in L'$. A word $v$ in $L$ is {\em good} (for $v'$) if it satisfies the
following properties. Let $\alpha'$ be the path in $C_{S'}(G)$ associated to $v'$, and $\alpha$ the path
in $C_{S}(G)$ associated to $v$, and think of the vertices $\alpha(i)$ as vertices in $C_{S'}(G)$.
Then $v$ is good if it satisfies properties (1) and (2) above (with $\alpha$ and $\alpha'$ in place
of $\gamma$ and $\gamma'$), and if furthermore the distance from $\overline{v}'$ to $\overline{v}$ is
at most $C_1$. Notice if $w_j'$ is a prefix of $w'$, then there is some prefix $w_i$ of $w$ as 
above that is good.

Let $\Gamma'$ be a digraph parameterizing $L'$, and let $\Gamma$ be a digraph parameterizing $L$
for which there is $d\phi$ from the states of $\Gamma$ to $\Z$ as in 
Definition~\ref{combable_function_definition}. Let $B$ denote the (set of vertices in the) 
ball of radius $C_1$ about $\id$
in $C_{S'}(G)$. For each $i$, let $\mu_i$ be the function whose domain is $B$, 
and for each $b \in B$ define $\mu_i(b)$ as follows:
\begin{enumerate}
\item{if there is no good word $v \in L$ (for $w'_i$) with $\overline{v}=\overline{w}'_ib$, set $\mu_i(b)$ equal to
an ``out of range'' symbol ${\tt{E}}$; otherwise:}
\item{set $\mu_i(b)$ equal to the tuple whose first entry is the value $\phi(\overline{w}'_ib)-\phi(\overline{w}'_i)$
and whose second entry is the state of $\Gamma$ after reading the word $v$}
\end{enumerate}
Note that given $b$, the word $v$ if it exists is unique, since $L \to G$ is a bijection. Moreover,
since $\phi$ is combable, there is a constant $C_3$ so that $\phi(\overline{w}'_ib)-\phi(\overline{w}'_i) \in [-C_3,C_3]$.
In other words, the range of $\mu_i$ is the {\em finite} set ${\tt{E}} \cup [-C_3,C_3] \times \Gamma$ 
where for the sake of brevity, and by abuse of notation,
we have denoted the set of states of $\Gamma$ by $\Gamma$ (this is the only point in the
argument where combability (as distinct from weak combability) is used). Finally, observe that $\mu_i(\id)$ is never equal
to ${\tt{E}}$, since there is a good $v$ for $w'_i$ with $\overline{v}=\overline{w}'_i$, for every $i$.

We claim that the function $\mu_i$ and the letter $w'(i+1)$ together determine the function $\mu_{i+1}$.
For each $b \in B$, the function $\mu_i$ tells us whether or not there is a good word $v$ with
$\overline{v} = \overline{w}'_ib$, and if there is, the state of $\Gamma$ after reading such a word.
There are finitely many words $x$ of length at most $C_2$ in $S^*$ which can be concatenated to $v$ to
obtain a word $vx \in L$, and the set of such $x$ depends only on the state of $\Gamma$ after reading $v$.
Each suffix $x$ determines a path $\beta$ in $C_S(G)$ and a discrete path $\overline{v}\beta(j)$ in
$C_{S'}(G)$. Whether the composition $vx$ is good for $w'_{i+1}$ or not depends only on the suffix $x$,
so the set of such suffixes can be determined for each $b \in B$, and the set of endpoints determines a
subset of vertices $b' \in B$ for which $\mu_{i+1}(b') \ne {\tt{E}}$. For each such $b'$, the state of $\Gamma$
after reading $vx$ can be determined, and the difference $\phi(\overline{vx})-\phi(\overline{v})$ can be determined by
the combability of $\phi$ with respect to $\Gamma$.

We claim further that the function $\mu_i$ and the letter $w'(i+1)$ together determine 
$\phi(\overline{w}'_{i+1}) - \phi(\overline{w}'_{i})$. For, there is a unique $b$ and good $v \in L$ with
$\overline{v} = \overline{w}'_ib$ which can be extended to $vx \in L$ with $\overline{vx}=\overline{w}'_{i+1}$.
Because we know the state of $\Gamma$ after reading $v$, we know $\phi(\overline{vx}) - \phi(\overline{v})$.
Moreover, we know $\phi(\overline{vx})-\phi(\overline{w}'_i)$, so we know $\phi(\overline{w}'_{i+1})-\phi(\overline{w}'_i)$
as claimed.

\medskip

From this data it is straightforward to construct a new $\Gamma''$ 
parameterizing $L'$ so that $d\phi$ exists as a function on (the states of) $\Gamma''$. 
Every word $w' \in L$ and every integer $i\ge 0$ determines a function $\mu_i$, and a state 
$\gamma'(i) \in \Gamma'$ (by abuse of notation, but in accordance with Convention~\ref{path_convention} we
denote the path in $\Gamma'$ associated to $w'$ by $\gamma'$), and thereby a triple
$(\mu_i,\gamma'(i),\phi(\overline{w}'_i) - \phi(\overline{w}'_{i-1}))$ (with the convention
$\phi(\overline{w}'_{-1})=0$). The vertices of $\Gamma''$ are
the set of all possible triples that occur, over all $w' \in L'$ and all $i$; note that this is a {\em finite} set.
The directed edges of $\Gamma''$ are the ordered pairs of triples that can occur for some $w'$ and
successive indices $i$, $i+1$. Evidently $\Gamma''$ parameterizes $L'$; moreover, $d\phi$ is the function
whose value on a triple as above is the last co-ordinate.
\end{proof}

The sum or difference of two (bi-)combable functions is (bi-)combable, and
the set of all (bi-)-combable functions on a fixed $G$ is therefore a
countably infinitely generated free Abelian summand of the 
Abelian group $\Z^G$ of all $\Z$-valued functions on $G$. 

\begin{example}
Word length with respect to a finite generating set $S_1$ is bicombable with respect to
$L$ where $L$ is a combing with respect to another finite generating set $S_2$
(compare with Example~\ref{word_length_bicombable}). 
\end{example}

\subsection{Quasimorphisms}\label{quasimorphism_section}

\begin{definition}
Let $G$ be a group. A {\em quasimorphism} is a real valued function $\phi$ on $G$
for which there is a least non-negative real number $D(\phi)$ called the {\em defect},
such that for all $a,b \in G$ there is an inequality
$$|\phi(a) + \phi(b) - \phi(ab)| \le D(\phi)$$
\end{definition}

In words, a function is a quasimorphism if it is a homomorphism up to a bounded error.

Quasimorphisms are related to stable commutator length and to $2$-dimensional bounded
cohomology. For a more substantial discussion, see
\cite{Bavard} or \cite{Calegari_scl_monograph}.

\begin{lemma}
A weakly combable quasimorphism is bicombable.
\end{lemma}
\begin{proof}
For any $g \in G$ and $s \in S$ we have 
$$|\phi(gs) - \phi(g)| \le |\phi(s)| + D(\phi)$$
and
$$|\phi(sg) - \phi(g)| \le |\phi(s)| + D(\phi)$$
Since $s$ is finite, the conclusion follows.
\end{proof}

Conversely, let $\phi$ be bicombable. Then we have the following lemma:

\begin{lemma}\label{subdivide_lemma}
Let $\phi$ be bicombable. Then there is a constant $C$ so that if $w \in L$ is expressed as
a product of subwords $w = uv$ then $|\phi(\overline{w}) - \phi(\overline{u}) - \phi(\overline{v})| \le C$.
\end{lemma}
\begin{proof}
Let $u' \in L$ be any word so that the paths in $\Gamma$ associated to $u$ and to $u'$
end at the same vertex. Then $\phi(\overline{w}) = \phi(\overline{u}) + \phi(\overline{u'v}) -
\phi(\overline{u'})$. Now choose $u'$ of bounded length, and notice that the difference of $\phi(\overline{u'v})$
and $\phi(\overline{v})$ is bounded because of the Lipschitz condition.
\end{proof}

Using this Lemma we can construct many bicombable quasimorphisms on $G$.

\begin{theorem}\label{bicombable_quasimorphisms_exist}
Let $G$ be hyperbolic, and let $S$ be a finite generating set for $G$ as a semigroup.
Let $\phi_S$ be the bicombable function which counts word length with respect to $S$, and
$\phi_{S^{-1}}$ the bicombable function which counts word length with respect to $S^{-1}$.
Then $\psi_S: = \phi_S - \phi_{S^{-1}}$ is a bicombable quasimorphism.
\end{theorem}
\begin{proof}
Since each of $\phi_S$ and $\phi_{S^{-1}}$ are bicombable, so is their difference.
Let $g,h$ be arbitrary. Let $w,u,v \in L$ correspond to $gh,g,h$ respectively.
Then we can write $u = u'x, v = yv'$ and $w = w_1w_2$ as words in $L$ so that
$d_L(y,x^{-1}) \le \delta$, $d_L(u',w_1) \le \delta$ and $d_R(v',w_n) \le \delta$ by $\delta$-thinness
of triangles. Now apply Lemma~\ref{subdivide_lemma}.
\end{proof}

\begin{remark}
Given an arbitrary bicombable function $\phi$, one can define $\psi$ by $\psi(g) = \phi(g) - \phi(g^{-1})$.
It is not necessarily true that $\psi$ is bicombable, but an argument similar to the proof of
Theorem~\ref{bicombable_quasimorphisms_exist} shows that $\psi$ is a quasimorphism.
\end{remark}

\subsection{Counting quasimorphisms}

Among the earliest known examples of quasimorphisms are counting quasimorphisms on free groups,
first introduced by Brooks \cite{Brooks}. Brooks' construction was
generalized to word-hyperbolic groups by Epstein-Fujiwara \cite{Epstein_Fujiwara}.

\begin{definition}
Let $G$ be a hyperbolic group with generating set $S$. Let $\sigma$ be an
oriented simplicial path in the Cayley graph $C_S(G)$ and let $\sigma^{-1}$ denote the
same path with the opposite orientation. A {\em copy} of $\sigma$ is a translate
$a\cdot \sigma$ for some $a \in G$. For $\alpha$ an oriented simplicial path in
$C_S(G)$, let $|\alpha|_{\sigma}$ denote the maximal number of disjoint copies of $\sigma$
contained in $\alpha$. For a directed path $\alpha$ in $C_S(G)$ from $\id$ to $a$,
define
$$c_\sigma(\alpha) = |a| - (\length(\alpha) - |\alpha|_\sigma)$$ and
define
$$c_\sigma(a) = \sup_\alpha c_\sigma(\alpha)$$
where the supremum is taken over all directed paths $\alpha$ in $C_S(G)$ from
$\id$ to $a$.

Define a {\em counting quasimorphism} to be a function of the form
$$\phi_\sigma(a):=c_\sigma(a) - c_{\sigma^{-1}}(a)$$
\end{definition}

\begin{remark}
There is no logical necessity to count disjoint copies of $\sigma$ in $\alpha$ rather than
all copies. The latter ``big'' counting function agrees with the convention of Brooks whereas
the former ``small'' counting function agrees with the convention of Epstein-Fujiwara.
\end{remark}

In the sequel, by convention assume that the length of $\sigma$ is at least $2$. From
the definition, $\phi_\sigma(g) = -\phi_\sigma(g^{-1})$ for all $g$.
A path $\alpha$ which realizes the supremum in the formula above is called a
{\em realizing path} for $a$.
Since the value of each term of the form $|\cdot|$, $\length(\cdot)$, $|\cdot|_\sigma$ is integral,
a realizing path exists for every $a$.

The crucial property of realizing paths for our applications is the following
quasigeodesity property:

\begin{lemma}[Epstein-Fujiwara \cite{Epstein_Fujiwara}]\label{uniform_estimate}
Let $\sigma$ be a path in $C_S(G)$ of length at least $2$,
and let $a \in G$. A realizing path for $a$ is a $K,\epsilon$-quasigeodesic in $C_S(G)$,
where $K$ and $\epsilon$ are universal.
\end{lemma}

Note if $G$ is $\delta$-hyperbolic with respect to some $S$, then a $K,\epsilon$-quasigeodesic
is within distance $C$ of an actual geodesic, where $C$ depends on $K,\epsilon,\delta$.
In other words, a realizing path (asynchronously) fellow-travels
any geodesic representative of the word.

\subsection{Greedy algorithm}

Let $\sigma$ be a string. Let $k_\sigma$ count the maximal number of disjoint copies of
$\sigma$ in a word $w$, and let $k'_\sigma$ count disjoint copies of $\sigma$ in $w$  by using
the {\em greedy algorithm}: that is, there is equality
$k'_\sigma(w) = k'_\sigma(v) + 1$ where $v$ is the word obtained from $w$ by deleting the prefix
up to and including the first occurrence of $\sigma$.

\begin{lemma}[Greedy is good]\label{greedy_is_best}
The functions $k_\sigma$ and $k'_\sigma$ are equal.
\end{lemma}
\begin{proof}
Let $w$ be a shortest word for which $k_\sigma(w)$ and $k'_\sigma(w)$ are not equal. By definition,
$k'_\sigma(w) < k_\sigma(w)$ and since $w$ is shortest, $k'_\sigma(w) = k_\sigma(w)-1$. Since
$w$ is shortest, the suffix of $w$ must be a copy of $\sigma$ which is counted by $k_\sigma$ but
not by $k'_\sigma$. Hence the greedy algorithm must count a copy of $\sigma$ which overlaps with this
copy. So deleting the terminal copy of $\sigma$ reduces the values of both $k_\sigma$ and $k'_\sigma$
by $1$, contrary to the hypothesis that $w$ was shortest.
\end{proof}

\subsection{Hyperbolic groups}\label{hyperbolic_subsection}

The purpose of this section is to prove that the Epstein-Fujiwara counting
quasimorphisms are bicombable. The proof has a number of structural similarities with the
proof of Lemma~\ref{combing_independent_lemma}. Note that it suffices to show that they
are weakly combable, since any quasimorphism is Lipschitz in both the $d_L$ and the $d_R$
metrics.

\begin{theorem}\label{counting_hyperbolic_combable}
Counting quasimorphisms on hyperbolic groups are bicombable.
\end{theorem}
\begin{proof}
Fix a hyperbolic group $G$ and a symmetric generating set $S$. Let $L$ be a
combing for $G$. Remember that this means that $L$ is a prefix-closed regular language
of geodesics in $G$ (with respect to the fixed generating set $S$) for which the
evaluation map is a bijection $L \to G$. 

Let $\sigma$ be a string, and $\phi_\sigma$ the associated counting quasimorphism. It suffices to show
that $c_\sigma$ is weakly combable. For the sake of convenience, abbreviate $c_\sigma$ to $c$
throughout this proof.

\medskip

Fix a word $w \in L$, and let $\gamma$ be the corresponding path in $C_S(G)$.
By Lemma~\ref{uniform_estimate}, a realizing path $\gamma'$ for $\overline{w}$ is a $K,\epsilon$ 
quasigeodesic, and therefore by the Morse Lemma, it satisfies the following two conditions (compare with
the proof of Lemma~\ref{combing_independent_lemma}):
\begin{enumerate}
\item{there is a constant $C_1$ so that every $\gamma(i)$ is within distance $C_1$ of some $\gamma'(j)$,
and every $\gamma'(j)$ is within distance $C_1$ of some $\gamma(i)$}
\item{there is a constant $C_2$ so that if $\gamma'(i)$ is within distance $C_1$ of $\gamma(j)$, then
$\gamma'(k)$ is within distance $C_1$ of $\gamma(j+1)$ for some $k$ with $|k-i|\le C_2$}
\end{enumerate}
Here distances are all measured in $C_S(G)$. Similarly, if $\alpha$ is a geodesic path, say that a
$K,\epsilon$ quasigeodesic path $\alpha'$ is good for $\alpha$ if they both start at $\id$, if their
endpoints are at most distance $C_1$ apart, and if they satisfy properties (1) and (2) above (with $\alpha$ and
$\alpha'$ in place of $\gamma$ and $\gamma'$).

Let $T$ denote the set of proper prefixes of $\sigma$. Given a (directed) 
path $\alpha \in C_S(G)$, first find all disjoint copies
of $\sigma$ in $\alpha$ obtained by the greedy algorithm, and let $p(\alpha) \in T$ be the biggest prefix of $\sigma$,
disjoint from these copies, that is a ``suffix'' of (the word corresponding to) $\alpha$.

Define a function $\mu_i$ as follows. Let $B$ denote the (set of vertices in the)
ball of radius $C_1$ in $C_S(G)$ about the identity. For each $i$, let $\mu_i$ be the function whose
domain is $B \times T$, and for each $b \in B$ and $t \in T$ define $\mu_i(b,t)$ as follows:
\begin{enumerate}
\item{if there is no good path $\alpha$ (for $\gamma_i$) with endpoint $\overline{w}_ib$ and $p(\alpha) = t$,
set $\mu_i(b,t)$ equal to an ``out of range'' symbol ${\tt{E}}$; otherwise:}
\item{set $\mu_i(b,t)$ equal to tuple whose first term is $|\overline{w}_ib|-|\overline{w}_i|$ (i.e.\ relative
distance to $\id$), and whose second
term is the difference $\nu_i(\overline{w}_ib,t) - c_\sigma(\overline{w}_i)$, where
$\nu_i(\overline{w}_ib,t)$ is the maximum of $c_\sigma$ on all good paths $\alpha$ as above with $p(\alpha) = t$}
\end{enumerate}
A good path ending at $\overline{w}_ib$ can be concatenated with a short path ending at $\overline{w}_i$ and
conversely, so there is a constant $C_3$ so that $\mu_i$ takes values in ${\tt{E}} \cup [-C_1,C_1]\times[-C_3,C_3]$, 
a finite set. Every good path ending at some $\overline{w}_{i+1}b'$ is obtained by concatenating a good path ending at
$\overline{w}_ib$ with a path of length at most $C_2$, and these can all be enumerated. The relative distance
to $\id$ on balls of uniform size can be determined as one moves along the geodesic $\gamma$,
by keeping track of only a bounded amount of data at each stage
(this is the heart of Cannon's argument that hyperbolic groups
are combable; for details see \cite{Cannon_automatic}). 
As in the proof of Lemma~\ref{combing_independent_lemma}, the function $\mu_i$ on $B \times T$
and the letter $w(i+1)$ together determine the function $\mu_{i+1}$, and the difference 
$c_\sigma(\overline{w}_{i+1})-c_\sigma(\overline{w}_i)$. From this it readily follows that $c_\sigma$ is
weakly combable, and (since it is evidently bilipschitz) therefore bicombable. The same is true of $c_{-\sigma}$,
and therefore of their difference $\phi_\sigma$.
\end{proof}

\section{Ergodic theory of bicombable functions}\label{ergodic_theory_section}

Let $G$ be a hyperbolic group with finite generating set $S$. Let $L$ be a combing
with respect to $S$, and $\Gamma$ a digraph parameterizing $L$. Graphs $\Gamma$
of this kind are not arbitrary, but are in fact quite constrained. The ergodic theory
of $\Gamma$ can be understood using the basic machinery of Markov chains. This material
is very well understood, but we keep the discussion self-contained and elementary
as far as possible. Basic references for this material include \cite{Berman_Plemmons}, 
\cite{Grinstead_Snell}, \cite{Kemeny_Snell}, \cite{Pollicott_Yuri} and \cite{Romanovskii}, 
and at a couple of points we refer the reader to these texts for proofs of standard facts.

\subsection{Almost semisimple directed graphs}

Let $\Gamma$ be a finite directed graph. Let $V$ be the real
vector space spanned by the vertices of $\Gamma$, and let $\langle\cdot,\cdot\rangle$
be the inner product on $V$ for which the vertices are an orthonormal basis.
We could and do think of $V$ as the space of real-valued functions on the vertices of $\Gamma$.
Let $V \otimes \C$ denote the corresponding vector space of complex-valued functions
on the vertices of $\Gamma$. Let $v_1$ denote the initial vertex, and enumerate the other
vertices somehow as $v_2,v_3,\cdots$.
For a vector $v \in V$, let $|v|$ denote the $L^1$ norm of $v$. That is,
$$|v| = \sum_i |\langle v,v_i\rangle|$$

\begin{definition}
The {\em transition matrix} of $\Gamma$ (also sometimes called the
{\em adjacency matrix}\/) is the $n \times n$
matrix $M$ whose $M_{ij}$ entry counts the number of directed edges from $v_i$ to $v_j$
(or is equal to $0$ if there are no such edges).
\end{definition}

\begin{lemma}\label{path_counting_lemma}
For any $v_i,v_j$ the number of directed paths in $\Gamma$ of length $n$
from $v_i$ to $v_j$ is $\langle v_i, M^n v_j \rangle = (M^n)_{ij}$.
\end{lemma}
\begin{proof}
This is true for $n=1$. Suppose by induction it is true for $n=m$. 
The number of paths of length $m+1$ from $v_i$ to $v_j$ is equal to
the sum over all $v_k$ of the number of paths of length $m$ from $v_i$ to $v_k$
times the number of paths of length $1$ from $v_k$ to $v_j$. By induction, this
is $\sum_k (M^m)_{ik} M_{kj} = (M^{m+1})_{ij}$.
\end{proof}

By Lemma~\ref{path_counting_lemma}, the number of paths in $\Gamma$ of length
$n$ starting at $v_1$ is $|(M^T)^nv_1|$, where $M^T$ denotes the transpose of $M$.
One could equally well adopt the notation $|v_1M^n|:=|(M^T)^nv_1|$ but we prefer the
more explicit notation (involving the transpose)
since there are many opportunities for left-right errors in what follows.

\begin{definition}\label{almost_definition}
A directed graph $\Gamma$ is {\em \almost} if it
satisfies the following properties.
\begin{enumerate}
\item{There is an {\em initial vertex} $v_1$}
\item{For every $i \ne 1$ there is a directed path in $\Gamma$ from $v_1$ to $v_i$}\label{outgoing_bullet}
\item{There are constants $\lambda > 1,K\ge 1$ so that 
$$K^{-1}\lambda^n \le |(M^T)^n v_1| \le K\lambda^n$$
for all positive integers $n$}\label{growth_bullet}
\end{enumerate}
\end{definition}

In what follows we will assume that $\Gamma$ is \almost.

\begin{lemma}\label{big_eigenvalues_diagonal}
Suppose $\Gamma$ is \almost. Then $\lambda$ is the largest real eigenvalue
of $M$. Moreover, for every eigenvalue $\xi$ of $M$ either $|\xi|<\lambda$
or else the geometric and the algebraic multiplicities of $\xi$ are equal.
\end{lemma}
\begin{proof}
It is convenient to work with $M^T$ in place of $M$. To prove the lemma, it suffices to
prove analogous facts about the matrix $M^T$.
Corresponding to the Jordan decomposition of $M^T$ over $\C$, let $\xi_1,\dots,\xi_m$
be the eigenvalues of the corresponding Jordan blocks (listed with multiplicity).
We consider $V \otimes \C$ with the hermitian inner product.

Bullet~(\ref{outgoing_bullet}) from Definition~\ref{almost_definition} implies that
for any $v_i$, there is an inequality $|(M^T)^nv_i| \le C_i|(M^T)^nv_1|$ for
some constant $C_i$. Since the $v_i$ span $V$, and since $V$ is finite dimensional,
there is a constant $C$ such that for all $w \in V$ with $|w|=1$
there is an inequality $|(M^T)^nw| \le C|(M^T)^nv_1|$.
This is true if $w \in V \otimes \C$ as well.

For each $i$, there is some $w_i \in V \otimes \C$ in the 
$\xi_i$-eigenspace for which
$$|(M^T)^nw_i| \ge \text{constant}\cdot n^{k-1}|\xi_i|^n$$
where $k$ is the dimension of the Jordan block associated to $\xi_i$.
Since $|(M^T)^nw_i| \le C|(M^T)^nv_1|$,
by bullet~(\ref{growth_bullet}) from Definition~\ref{almost_definition},
either $|\xi_i|<\lambda$ or $|\xi_i|=\lambda$ and $k=1$.

By the Perron-Frobenius theorem for non-negative matrices (see e.g. \cite{Pollicott_Yuri}, \S~3.2 pp.~23--26), 
$M$ has a largest real eigenvalue $\lambda'$ such that $|\xi|\le \lambda'$ for all eigenvalues
$\xi$. We must have $\lambda' = \lambda$ by the estimates above.
\end{proof}

Note that $M$ and $M^T$ have the same set of eigenvalues, with the same multiplicities. Since every
eigenvalue of $M$ with largest (absolute) value has geometric multiplicity equal to its algebraic multiplicity,
the same is true of $M^T$.

For any vector $v \in V$,
decompose $v = \sum_\xi v(\xi)$ into the components in the generalized
eigenspaces of the eigenvalues $\xi$. 
Note that $(M v)(\xi)=M(v(\xi))$, so we write it
as $M v (\xi)$. Since any two norms on
$V \otimes \C$ are equivalent, there is a constant
$K>1$ such that
$$K^{-1} \le \frac {|M^nv|} {\sum_\xi|M^nv(\xi)|} \le K$$
providing $M^n v \ne 0$.

\begin{lemma}\label{limit_exists}
For any vector $v \in V$,
the following limit
$$\rho(v) := \lim_{n \to \infty} n^{-1} \sum_{i = 0}^{n} \lambda^{-i} M^i v$$
exists and is equal to $v(\lambda) \in V$.
\end{lemma}
\begin{proof}
We suppress $v$ in the notation that follows.
For each eigenvalue $\xi$ define
$$\rho_n(\xi) = n^{-1} \sum_{i = 0}^{n} \lambda^{-i} M^i v(\xi)$$
and set $\rho_n = \sum_{\xi} \rho_n(\xi)$. With this notation, 
$\rho = \lim_{n \to \infty} \rho_n$,
and we want to show that this limit exists, and is equal to $v(\lambda)$.

By Lemma~\ref{big_eigenvalues_diagonal},
for each $\xi$, either $|\xi|<\lambda$ or $v(\xi)$ is
a $\xi$-eigenvector. In the first case, $\rho_n(\xi) \to 0$. In the second case,
either $\xi = \lambda$, or else the vectors $\lambda^{-i}M^iv(\xi)$ become equidistributed
in the unit circle in the complex line of $V\otimes \C$ spanned by $v(\xi)$.
It follows that $\rho_n(\xi) \to 0$ unless $\xi = \lambda$.
So $n^{-1}\sum_{0 \le i \le n} \lambda^{-i}M^iv(\xi) \to 0$ unless
$\xi=\lambda$, in which case $\rho_n(\lambda) = v(\lambda)$ is constant. 
\end{proof}

The same argument as Lemma~\ref{limit_exists} implies

\begin{lemma}\label{transpose_limit_exists}
For any vector $v \in V$, the following limit
$$\ell(v): = \lim_{n \to \infty} n^{-1} \sum_{i = 0}^{n} \lambda^{-i} (M^T)^i v$$
exists, and is equal to the projection of $v$ onto the left $\lambda$-eigenspace of $M$.
\end{lemma}

For any $v_i$, the partial sums $\rho_n(v_i)$ are non-negative real vectors 
so if $v$ is non-negative and real, so are $\rho(v)$ and $\ell(v)$. 

\begin{lemma}\label{transfer_formula}
For any $v,w \in V$, there is an identity
$$\langle \ell(v), w \rangle =  \langle \ell(v), \rho(w) \rangle 
= \langle v, \rho(w) \rangle$$
\end{lemma}
\begin{proof}
Let $\pi$ denote projection onto the (right) $\lambda$-eigenspace, so $\pi v = \rho(v)$ and
$\ell(v) = \pi^Tv$. Since $\pi$ is a projection, it is idempotent; i.e. $\pi \circ \pi = \pi$,
and similarly for $\pi^T$.
Hence
$$\langle\pi^T v,w\rangle = \langle \pi^T \pi^T v,w\rangle = \langle \pi^T v, \pi w \rangle 
= \langle v, \pi \pi w \rangle = \langle v, \pi w \rangle$$
\end{proof}

\subsection{Stationary Markov chains}\label{stationary_markov_chain_subsection}

We would like to be able to discuss the {\em typical} behavior of an infinite
(or sufficiently long) path in $\Gamma$. The standard way to do this
is to use the machinery of {\em Markov chains}. A {\em stationary Markov chain} is
a random process $x_0, x_1,x_2,\cdots$ where each $x_i$ is in one of a finite
set of states $V$. The process starts in one state, and successively moves
from one state to another. If the chain is in some state $v_i \in V$ then
the probability that it moves to state $v_j \in V$ at the next step is some
number $N_{ij}$ which depends only on $i$ and $j$; in other words, given
the present state, the future states are independent of the past states. The
starting state is determined by some initial probability distribution $\mu$.
See \cite{Grinstead_Snell} Chapter~11 for an introduction to the theory
of Markov chains.

\medskip

Our digraph $\Gamma$ determines a stationary Markov chain as follows. 
For each $n$, let $X_n$ denote the set of paths in $\Gamma$ of length $n$, starting at an
arbitrary vertex, and $Y_n$ the set of paths of length $n$ starting at the initial vertex.
A path of length $0$ is a vertex.
Restricting to a prefix defines surjections $X_{n+1} \to X_n$ and $Y_{n+1} \to Y_n$ for
all $n$. Giving each $X_n,Y_n$ the discrete topology, we obtain inverse limits $X$ and $Y$
which denote the space of infinite
paths starting at an arbitrary vertex, and infinite paths starting at the initial vertex
respectively. 

Topologically, $X$ and $Y$ are Cantor sets. The projections $X \to X_n$ and $Y \to Y_n$ pull back
points to open subsets called {\em cylinders}, which generate the Borel $\sigma$-algebras
of $X$ and $Y$.

\begin{definition}
The {\em shift map} $S:X \to X$ takes an infinite path to its suffix which is the
complement of the initial vertex.
\end{definition}

It is bothersome to use the same letter $S$ both for the shift map and for a generating set
for $G$, but (unfortunately) both notations are standard, and in any case the meaning should always be clear from context.

Let $\1$ denote the constant function on $\Gamma$ taking the value $1$ on every vertex.
Define a matrix $N$ by
$$N_{ij} = \frac {M_{ij}\rho(\1)_j} {\lambda \rho(\1)_i}$$
if $\rho(\1)_i \ne 0$ and otherwise define $N_{ii}=1$ and $N_{ij} = 0$ for $i \ne j$.
Define a measure $\mu$ on the vertices of $\Gamma$ by
$$\mu_i = \rho(\1)_i\ell(v_1)_i$$
where subscripts denote vector and matrix components.

\begin{lemma}\label{basic_properties}
The matrix $N$ is a stochastic matrix (i.e. it is non-negative, and for all $i$
it satisfies $\sum_j N_{ij}=1$) and preserves the measure $\mu$.
\end{lemma}
\begin{proof}
For any $i$ not in the support of $\rho(\1)$ we have $\sum_j N_{ij} = 1$ by fiat.
Otherwise we have
$$\sum_j N_{ij} = \sum_j \frac {M_{ij} \rho(\1)_j} {\lambda \rho(\1)_i} = \frac {(M\rho(\1))_i}
{\lambda \rho(\1)_i} = 1$$
This shows $N$ is a stochastic matrix. To see that it preserves $\mu$, we calculate
\begin{align*}
\sum_i \mu_iN_{ij} &= \sum_i \rho(\1)_i\ell(v_1)_i \frac {M_{ij} \rho(\1)_j} {\lambda \rho(\1)_i} \\
&= \frac 1 \lambda \rho(\1)_j \sum_i \ell(v_1)_i M_{ij} \\
&= \rho(\1)_j \ell(v_1)_j = \mu_j
\end{align*}
Hence $N$ preserves $\mu$.
\end{proof}

We scale $\mu$ to be a probability measure. By abuse of notation, we also denote this
probability measure by $\mu$. This probability measure is preserved by $N$.

There is an associated probability measure on each $X_n$, which we also denote by $\mu$, where
$$\mu(v_{i_0}v_{i_1}\cdots v_{i_n}) = \mu(v_{i_0})N_{i_0i_1}N_{i_1i_2}\cdots N_{i_{n-1}i_n}$$
This defines a measure on each cylinder (i.e. preimage of an element of $X_n$ under $X \to X_n$)
and therefore a probability measure $\mu$ on $X$.
Lemma~\ref{basic_properties} implies that $\mu$ on $X$
is invariant under the shift map $S$; i.e. $\mu(A) = \mu(S^{-1}(A))$ for all measurable
$A \subset X$. We call $\mu$ the {\em stationary measure}.

\begin{lemma}\label{support_of_measure}
Suppose $\Gamma$ is almost semisimple, and $\mu$ is the stationary measure on $\Gamma$ for $N$.
The support of $\mu$ is equal to the union of the maximal recurrent components of $\Gamma$
whose adjacency matrix has biggest eigenvalue $\lambda$. Furthermore, if $C,C'$ are distinct
components in the support of $\mu$, there is no directed path from $C$ to $C'$
\end{lemma}
\begin{proof}
Recall that the Landau notation $f(x) = \Theta(g(x))$ for non-negative functions $f,g$
of $x$ means that there are positive constants $k_1,k_2$ so that 
$$k_1 g(x) \le f(x) \le k_2 g(x)$$
and the Landau notation $f(x) = o(g(x))$ means that $f(x)/g(x) \to 0$ as $x \to \infty$.

Since $\mu_i = \rho(\1)_i \ell(v_1)_i$, a vertex $v_j$ is in the support of $\mu$ if and only if 
the number of paths from $v_1$ to $v_j$ of length $\in [n-k,n+k]$ is $\Theta(\lambda^n)$
for some sufficiently big constant $k$, and the number of outgoing paths from $v_j$ of
length $\in [n-k,n+k]$ is also $\Theta(\lambda^n)$. Associated to a recurrent component
$C$ of $\Gamma$ one can form the adjacency matrix of the subgraph $C$, which has a biggest
real eigenvalue $\xi(C)$, by the Perron-Frobenius Theorem. Suppose for every directed
path $\gamma$ from $v_1$ to $v_j$ and for every recurrent component $C$ which intersects
$\gamma$ we have $\xi(C) < \lambda$. Then it is easy to see that the number of paths
from $v_1$ to $v_j$ of length $\in [n-k,n+k]$ is $o(\lambda^n)$. Similarly, if
for every directed path $\gamma$ starting at $v_j$ and every recurrent component $C'$ which
intersects $\gamma$ we have $\xi(C') < \lambda$ then the number of outgoing paths from
$v_j$ of length $\in [n-k,n+k]$ is $o(\lambda^n)$. Hence $v_j$ is in the support of
$\mu$ if and only if there are recurrent components $C,C'$ with $\xi(C) = \xi(C') = \lambda$
such that there is a path from $v_1$ to $v_j$ intersecting $C$, and an outgoing path
from $v_j$ intersecting $C'$.
Note that $v_j \in C = C'$ is explicitly allowed here. If $C \ne C'$ but there is
a path from $C$ to $C'$, then the number of paths of length $n$ which start in $C$ and
end in $C'$ has order $\Theta(\sum_i \lambda^i \lambda^{n-i}) = \Theta(n \lambda^n)$, contrary to the
hypothesis that $\Gamma$ is almost semisimple. It follows necessarily that $C = C'$.
\end{proof}

The construction of the measure $\mu$ and the stochastic matrix $N$ might seem
unmotivated for the moment. In the next section we shall provide a more geometric
interpretation of these objects.

\subsection{Word-hyperbolic groups and Patterson-Sullivan measures}

We now relate this picture to the theory of hyperbolic groups.
Let $G$ be a non-elementary hyperbolic group with generating set $S$.

\begin{definition}
The {\em Poincar\'e series} of $G$ is the series
$$\zeta_G(s) = \sum_{g \in G} e^{-s|g|}$$
\end{definition}

This series diverges for all sufficiently small $s$, and converges for all sufficiently
large $s$. The {\em critical exponent} is the supremum of the values of $s$ for which
the series diverges. 

\begin{theorem}[Coornaert, \cite{Coornaert} Thm.~7.2]\label{Coornaert_theorem}
Let $G$ be a non-elementary hyperbolic group with generating set $S$.
Let $G_n$ be the set of elements of word length $n$.
Then there are constants $\lambda > 1,K \ge 1$ so that
$$K^{-1}\lambda^n \le |G_n| \le K\lambda^n$$
for all positive integers $n$.
\end{theorem}

It follows from Theorem~\ref{Coornaert_theorem} that the critical exponent of the
Poincar\'e series is equal to $\log(\lambda)$, and the series $\zeta_G(\log(\lambda))$ diverges.

\begin{example}
Not every group with a geodesic combing satisfies such an inequality. For example,
let $G = F_2 \times F_2$ with the standard generating set. Then we can estimate
$$K^{-1} n 3^n \le |G_n| \le K n 3^n$$
for suitable $K >0$.
On the other hand, the language of words $uv$ where $u$ is a reduced word in
the first $F_2$ factor and $v$ is a reduced word in the second factor defines a
geodesic combing.
\end{example}

For each $n$, let $\nu_n$ be the probability measure on $G$ defined by
$$\nu_n = \frac {\sum_{|g| \le n} \lambda^{-|g|}\delta_g } {\sum_{|g| \le n} \lambda^{-|g|}}$$
where $\delta_g$ is the Dirac measure on the element $g$. The measure $\nu_n$ extends
trivially to a probability measure on $G \cup \partial G$, where $\partial G$ denotes
the Gromov boundary of $G$.

\begin{definition}\label{Patterson_Sullivan_definition}
A weak limit $\nu:= \lim_{n \to \infty} \nu_n$ is a {\em Patterson-Sullivan} measure
associated to $S$.
\end{definition}

In fact, the limit in Definition~\ref{Patterson_Sullivan_definition} exists, and an explicit formula
for a closely related measure $\widehat{\nu}$ (which differs from $\nu$ only by scaling) will be given shortly.
Note that the support of $\nu$ is contained in $\partial G$, since the Poincar\'e
series diverges at the critical exponent $\lambda$.

\medskip

Let $L$ be a combing of $G$ with respect to $S$, and let $\Gamma$ be a digraph.
From Coornaert's Theorem we immediately deduce

\begin{lemma}
Let $\Gamma$ be as above. Then $\Gamma$ is \almost.
\end{lemma}

A slightly better normalization of the measures $\nu_n$ are the sequence of measures $\widehat{\nu}_n$ 
defined by the formula
$$\widehat{\nu}_n = \frac 1 n \sum_{|g| \le n} \lambda^{-|g|}\delta_g$$
and let $\widehat{\nu}$ be a weak limit (supported on $\partial G$).

\begin{lemma}\label{nu_and_hat_nu}
Any two weak limits $\nu,\widehat{\nu}$ on $\partial G$ as above are absolutely continuous
with respect to each other, and satisfy $1/K \le d\widehat{\nu}/d\nu \le K$ for some $K\ge 1$.
\end{lemma}
\begin{proof}
The lemma is true for each $\widehat{\nu}_n,\nu_n$ by Theorem~\ref{Coornaert_theorem}, and
therefore it is also true for their limits.
\end{proof}

The measure $\widehat{\nu}$ is not a probability measure in general, though it is finite.

\medskip

The group $G$ acts on itself by left multiplication. This action extends continuously
to a left action $G \times \partial G \to \partial G$.
Patterson-Sullivan measures enjoy a number of useful properties, summarized in the
following theorem.

\begin{theorem}[Coornaert, \cite{Coornaert} Thm. 7.7]\label{PS_Radon_Nikodym_lemma}
Let $\nu$ be a Patterson-Sullivan measure. The action of $G$ on $\partial G$ preserves
the measure class of $\nu$. Moreover, the action of $G$ on $\partial G$ is ergodic
with respect to $\nu$.
\end{theorem}

Here we say that a group action which preserves a measure class is {\em ergodic} 
if for all subsets $A,B \subset \partial G$ with positive
$\nu$-measure, there is $g \in G$ with $\nu(gA \cap B) > 0$.
By Lemma~\ref{nu_and_hat_nu}, the action of $G$ on $\partial G$ preserves the measure
class of $\widehat{\nu}$ and acts ergodically.

\begin{remark}
In fact, Coornaert shows that there is some constant $K \ge 1$ such that for any $s \in S$
there is an inequality
$$K^{-1} \le \frac {d (s_*\nu)} {d\nu} \le K$$
i.e. the action is ``quasiconformal''. It follows that the same is true for the $\widehat{\nu}$
measure, with the same constant $K$. We will not use this fact in this paper.
\end{remark}

We now relate the measure $\widehat{\nu}$ to the measure $\mu$ and the stochastic matrix
$N$ defined in \S~\ref{stationary_markov_chain_subsection}. 

There is a natural map $Y_n \to G$ for each $n$, taking a path in $\Gamma$ to the
endpoint of the corresponding geodesic path in $G$ which starts at $\id$. 
The measures $\widehat{\nu}_n$ determine a measure
on $Y$ as follows. 

For each $g \in G$, let $\cone(g)$ denote the set of $g' \in G$ such that
the geodesic for $g'$ in the combing passes through $g$. Let $p:Y \to Y_n$ take an
infinite path to its prefix of length $n$. If $y \in Y_n$ corresponds to
$g \in G$, define 
$$\widehat{\nu}(p^{-1}(y)) = \lim_{n \to \infty} \widehat{\nu}_n(\cone(g))$$
Since the cylinders $p^{-1}(y)$ define the Borel $\sigma$-algebra on $Y$, this defines a measure
$\widehat{\nu}$ on $Y$. There is a map $Y \to \partial G$ taking an infinite path to the endpoint
of the corresponding geodesic in $G$. The pushforward
of $\widehat{\nu}$ on $Y$ under this map agrees with
$\widehat{\nu}$ on $\partial G$.

We give an explicit formula for $\widehat{\nu}_m(\cone(g))$ for any $g$. By definition,
for any $g \in G_n$ and any $m \ge n$ we have
$$\widehat{\nu}_m(\cone(g)) = \frac 1 m \sum_{\substack{h \in \cone(g) \\ |h| \le m}} \lambda^{-|h|}$$
If $v_g \in \Gamma$ is the vertex corresponding to the endpoint of the path for the
element $y \in Y_n$ corresponding to $g$, then we can rewrite this as
$$\widehat{\nu}_m(\cone(g)) = \frac 1 m \lambda^{-n} \sum_{i \le m-n} \lambda^{-i} \langle (M^T)^i v_g, \1 \rangle$$
taking a limit as $m \to \infty$, we get
$$\widehat{\nu}(p^{-1}(y)) = \lambda^{-n}\langle \ell(v_g),\1 \rangle = \lambda^{-n} \langle v_g,\rho(\1)\rangle = \lambda^{-n} \rho(\1)_g$$

\medskip

Measures are pushed forward by maps, so $S_*\widehat{\nu}$ is the measure on $X$ satisfying
$$S_*\widehat{\nu}(B) = \widehat{\nu}(S^{-1}(B))$$
for any Borel $B \subset X$, and similarly for positive powers $S^i$.

The subset $Y \subset X$ contains the support of $\widehat{\nu}$. By convention, $v_1$ has no incoming edges,
so $Y$ is disjoint from $SY$, and therefore $\widehat{\nu}$ is not invariant under $S$. However, the limit
$$\lim_{n \to \infty} \sum_{i=1}^n \frac 1 n S^i_* \widehat{\nu}$$
is manifestly $S$ invariant. The fact that this limit exists is contained in the proof of the
next Lemma.

\begin{lemma}
The measure $\lim_{n \to \infty} \sum_{i=1}^n \frac 1 n S^i_* \widehat{\nu}$ on $X$ 
is equal to the measure $\mu$ on $X$ defined on cylinders using the measure
$\mu$ on $\Gamma$ and $N$ as in \S~\ref{stationary_markov_chain_subsection}.
\end{lemma}
\begin{proof}
We identify $X_0 = \Gamma$, and observe that the measure $\mu$ on $\Gamma$ agrees
with the measure $\mu$ on cylinders $p^{-1}(v)$ of $X$ for $v \in \Gamma$. For
each vertex $v_i$ of $\Gamma$ we can calculate
$$S^n_* \widehat{\nu}(p^{-1}(v_i)) = \lambda^{-n} \sum_{\substack{y \in Y_n \\ S^n y = v_i}}
\langle v_i,\rho(\1)\rangle$$
On the other hand, the number of $y \in Y_n$ with $S^ny = v_i$ is exactly equal to the
number of directed paths in $\Gamma$ of length $n$ which start at $v_1$ and
end at $v_i$, which is $\langle v_1,M^nv_i \rangle$. Hence
\begin{align*}
\lim_{n \to \infty} \sum_{i=1}^n \frac 1 n S^i_* \widehat{\nu}(p^{-1}(v_i)) 
&= \lim_{n \to \infty} \langle v_i,\rho(\1) \rangle \langle v_1, \frac 1 n \sum_{i \le n} \lambda^{-n} M^n v_i \rangle \\
& = \langle v_i,\rho(\1)\rangle \langle v_1,\rho(v_i)\rangle = \rho(\1)_i\ell(v_1)_i = \mu_i
\end{align*}
This shows that the measures agree on the cylinders $p^{-1}(v_i)$ for $v_i \in \Gamma = X_0$.
By the defining property of $\mu$ on $X$, it suffices to show that for every $y \in Y$ whose
$n$th vertex is $v_i$, the $\widehat{\nu}$ probability that the $(n+1)$st vertex is $v_j$ is
$N_{ij}$. If $y_n$ is a prefix of $y$ of length $n$ whose last vertex is $v_i$,
and $g \in G_n$ is the corresponding element, we need to calculate
$$\lim_{m \to \infty} \frac {\sum_h \widehat{\nu}_m(\cone(h))} {\widehat{\nu}_m(\cone(g))}$$
where the sum is taken over all $h \in G_{n+1} \cap \cone(g)$ corresponding to $y' \in Y_{n+1}$
whose last vertex is $v_j$.
For each such $g$, the number of $h$ is $M_{ij}$ so this ratio is equal to 
$M_{ij} \rho(\1)_j/\lambda\rho(\1)_i$ which equals
$N_{ij}$ as claimed.
\end{proof}

In fact we can be even more precise about the relationship
between $\widehat{\nu}$ and $\mu$. If $y \in Y_n$ is a prefix whose
last vertex projects to some $v \in \Gamma$ in the support of $\mu$,
then $S^n$ maps the cylinder $p^{-1}(y)$ of $Y$ to the cylinder $X_v$ of $X$ consisting
of paths with initial vertex $v$, and from the definitions,
the pushforward of $\widehat{\nu}|_{p^{-1}(y)}$
to $X$ under $S^n$ is proportional to $\mu|_{X_v}$.

\subsection{Central Limit Theorem for stationary Markov chains}

The stationary measure $\mu$ on $\Gamma$ defined in \S~\ref{stationary_markov_chain_subsection}
has support equal to a finite union of maximal recurrent components of $\Gamma$,
by Lemma~\ref{support_of_measure}. Label these components $C^i$, and let $\mu^i$ be the
probability measure on $C^i$ which is proportional to $\mu|_{C^i}$. The matrix $N$
preserves $\mu^i$, so it makes sense to talk about a random walk on $C^i$ with respect to $N$.

Recall that $d\phi$ exists as a function on the vertices of $\Gamma$. The stationary
Markov chain defined by the matrix $N$ defines a random sequence $x_0,x_1,x_2,\cdots$ of vertices 
in $C^i$ where $x_0$ is chosen randomly from the vertices of $C^i$ with
respect to the measure $\mu^i$, and at each subsequent stage if $x_i$ corresponds to a vertex
$v_j$, the probability that $x_{i+1}$ will correspond to the vertex $v_k$ is equal to
$N_{jk}$. A sequence $x_0,x_1,x_2,\cdots$ obtained in this way is called a {\em random walk}
on $C^i$. If $X_n^i$ denotes the subset of $X_n$ consisting of paths of length $n$
which start at some vertex in $C^i$, then the measure $\mu$ on $X_n$ restricted to
$X_n^i$ can be scaled to a probability measure $\mu^i$ on $X_n^i$. In this notation,
a random walk on $C^i$ of length $n$ corresponds to a random element of $X_n^i$.

Let $\overline{S}^i_n$ be the sum of the values of $d\phi$ on a random walk on $C^i$ of length
$n$. Technically, $\overline{S}^i_n$ is a {\em distribution} on $\Z$. There is a map
$X_n^i \to \Z$ taking each walk $x_0,x_1,x_2,\cdots$ to the sum of $d\phi$ on the corresponding
vertices in $\Gamma$. The pushforward of the measure $\mu^i$ on $X_n^i$ is the distribution
$\overline{S}^i_n$. In this context, the Central Limit Theorem for stationary Markov chains, due essentially
to Markov, takes the following form:

\begin{theorem}[Markov]\label{clt_component}
Let $E^i$ be the integral of $d\phi$ with respect to the probability measure $\mu^i$.
Then there is some $\sigma^i \ge 0$ for which
$$\lim_{n \to \infty} {\bf P} \left( r \le \frac {\overline{S}^i_n - nE^i} {\sqrt{n(\sigma^i)^2}} \le s\right)
= \frac 1 {\sqrt{2\pi}} \int_r^s e^{-x^2/2} dx$$
where ${\bf P}$ denotes probability with respect to the measure $\mu^i$.
\end{theorem}

Said another way, there is a convergence in the sense of distribution
$$n^{-1/2}(\overline{S}^i_n - nE^i) \to N(0,\sigma^i)$$
where $N(0,\sigma^i)$ denotes a normal (i.e. Gaussian) 
distribution with mean $0$ and standard deviation $\sigma^i$.
The standard deviation $\sigma^i$ can be computed easily from $\mu$ and $N$, and is an algebraic
function of the entries. Hence in particular, $E^i$ and $\sigma^i$ are algebraic and effectively
computable.

For a proof, see e.g. \cite{Romanovskii}, p.~231, and see Chap.~4, \S~47 for a derivation of
a number of formulae for $\sigma^i$.

\subsection{Ergodicity at infinity and CLT for bicombable functions}

To derive a central limit theorem for $\phi$, we need to be able to compare means $E^i$ and
standard deviations $\sigma^i$ associated to different components $C^i$. To do this we introduce
the idea of a {\em typical path} (with respect to the function $\phi$).

For any real number $r$, let $\delta(r)$ denote the probability measure on $\R$ consisting of
an atom of mass $1$ concentrated at $r$.
Let $\gamma$ be an infinite geodesic ray in $G$. For some fixed real number $E$ and
for integers $n,m$ we can define the distribution
$$\omega(n,m)(\gamma) := \sum_{i=1}^m \frac 1 m \; \delta \left( (\phi(\gamma_{i+n}) - \phi(\gamma_i) - nE)n^{-1/2}\right)$$

\begin{definition}
A geodesic ray $\gamma$ is $E,\sigma$-typical if the limit
$$\omega(\gamma):=\lim_{n \to \infty} \lim_{m \to \infty} \omega(n,m)(\gamma)$$
exists in the sense of distribution, and is equal to $N(0,\sigma)$.
An element $y \in Y$ is $E,\sigma$-typical if the corresponding geodesic ray in $G$
is $E,\sigma$-typical.
\end{definition}

We parse this condition as follows. For each segment of $\gamma$ of length $n$ (contained between $\gamma(0)$ and
$\gamma(n+m)$), we look at the difference of $\phi$ on the endpoints. In this way we obtain a set of $m$ numbers,
which we think of as an probability measure on $\R$, by weighting each number equally. We then translate the
origin by $nE$, and rescale the $x$-axis by a factor of $n^{-1/2}$, and consider the resulting probability measure
$\omega(n,m)(\gamma)$. For fixed $n$, we consider the limit $\omega(n)(\gamma)$ (if it exists); this distribution is
the (rescaled) difference of $\phi$ on the endpoints of a ``random'' segment of $\gamma$ of length $n$. Finally, we
take the limit as $n \to \infty$ (the rescaling is done so that this limit typically exists). The result $\omega(\gamma)$
is the difference of $\phi$ on the endpoints of a ``long random'' segment of $\gamma$, translated and scaled.

For each component $C^i$, let $Y^i$ be the subset of $Y$ consisting of paths which eventually
enter the component $C^i$ and never leave. Note that $Y^i$ and $Y^j$ are disjoint if $i \ne j$.

\begin{lemma}
The measure $\widehat{\nu}(Y - \cup_i Y^i) = 0$. Moreover, $\widehat{\nu}$ a.e. $y \in Y^i$ are
$E^i,\sigma^i$-typical, where $E^i,\sigma^i$ are as in Theorem~\ref{clt_component}.
\end{lemma}
\begin{proof}
The number of paths in $\Gamma$ of length $n$
which never enter a component $C$ with $\xi(C)=\lambda$ is $o(\lambda^n)$, so from the
definition of $\widehat{\nu}$ it follows that $\widehat{\nu}$ a.e. $y \in Y$ enters some $C^i$.
If a path leaves some component with $\xi = \lambda$ it can never enter another one
with $\xi = \lambda$, so $\widehat{\nu}$ a.e. $y \in Y$ which enter some $C^i$ never leave.

\medskip

We now prove the second statement of the Lemma. The following proof was suggested by Shigenori Matsumoto.

We fix the notation below for the course of the Lemma (the reader should be warned that it is slightly
incompatible with notation used elsewhere; this is done to avoid a proliferation of subscripts). Let $C_i$
be a component of $\Gamma$ with Perron-Frobenius eigenvalue $\lambda$. Let $Y_i$ be the set of infinite paths in
$\Gamma$ starting at $v_1$ that eventually stay in $C_i$, and let $X_i$ be the set of infinite paths in $C_i$.
There is a measure $\widehat{\mu}_i$ on $X_i$ obtained by restricting $\mu$ on $X$. The measure $\widehat{\mu}_i$
is determined by a stationary measure $\mu_i$ on $C_i$ and a transition matrix $N(i)$, by restricting $\mu$
and $N$. The measure $\mu_i$ is stationary for $N(i)$ (i.e.\ $\mu_i^T N(i) = \mu_i^T$), so $\widehat{\mu}_i$
is shift invariant. Since $C_i$ is a component, $\mu_i^T$ is the only eigenvector of $N(i)$ with eigenvalue $1$,
so $\mu_i$ is extremal in the space of stationary measures. Therefore by the random ergodic theorem 
(see e.g. \cite{Pollicott_Yuri}, Chapter~10) the measure $\widehat{\mu}_i$ on $X_i$ is ergodic.

Now, there is a subset $X_i^*$ of $X_i$ of full measure such that for all $\gamma \in X_i^*$
$$\frac 1 m \sum_0^m \delta_{S^k\gamma} \to \widehat{\mu}_i$$
in the weak$^*$ topology. On the other hand, on $Y_i$ there is a measure $\widehat{\nu}_i$ which is the restriction of
$\widehat{\nu}$. Define $q:Y_i \to X_i$ by
$$q(\gamma) = S^{n(\gamma)}(\gamma)$$
where $n:Y_i \to \mathbb{N}$ satisfies the following condition. Let $\pi:X_i \to C_i$ take each infinite walk to
its initial vertex. Choose $n$ so that $\pi \circ q:Y_i \to C_i$ sends the measure $\widehat{\nu}_i$ on $Y_i$ to
a measure $\mu_q$ on $C_i$ of full support. The measure $q_*\widehat{\nu}_i$ on $X_i$ is obtained from an initial
measure $\mu_q$ and the transition matrix $N(i)$ as in \S~\ref{stationary_markov_chain_subsection}; 
consequently $q_*\widehat{\nu}_i$ and $\mu_i$ are quasi-equivalent
(i.e.\ each is absolutely continuous with respect to the other).

It follows that $Y_i^*:=q^{-1}(X_i^*)$ has full measure with respect to $\widehat{\nu}_i$, and if $\gamma \in Y_i^*$,
then
$$\frac 1 m \sum_0^m \delta_{S^k\gamma} \to \widehat{\mu}_i$$
By Theorem~\ref{clt_component} it follows that the geodesic ray in $G$ associated to any $\gamma \in Y_i^*$ is $E^i,\sigma^i$-typical, and the lemma
is proved.
\end{proof}

Up to this point, we have only used the weak combability of $\phi$. The next Lemma
requires bicombability.

\begin{lemma}\label{typical_compare_lemma}
Let $\gamma$ be an $E,\sigma$-typical geodesic ray in $G$. If $\phi$ is combable, and if
$\gamma'$ is a geodesic ray with the same endpoint in $\partial G$ as $\gamma$, then $\gamma'$ is also
$E,\sigma$-typical. If $\phi$ is bicombable then for any $g \in G$ the translate $g\gamma$
is $E,\sigma$-typical.
\end{lemma}
\begin{proof}
Let $\gamma$ and $\gamma'$ have the same endpoint. Then there is a constant $C$ such that
$d_L(\gamma_i,\gamma'_i) \le C$ and therefore $|\phi(\gamma_i) - \phi(\gamma'_i)| \le K$ for some $K$
independent of $i$. This shows that $\gamma'$ is $E,\sigma$-typical if $\gamma$ is. Similarly,
if $g \in G$ then $d_R(g\gamma_i,\gamma_i) \le C$ and therefore $|\phi(g\gamma_i) - \phi(\gamma_i)| \le K$
for some $K$ independent of $i$.
\end{proof}

For each component $C^i$, let $\partial^i G$ denote the image of the $E^i,\sigma^i$-typical
$y \in Y^i$ under $Y \to \partial G$. Note that $\widehat{\nu}(\partial^i G) > 0$ for all $i$.
By Theorem~\ref{PS_Radon_Nikodym_lemma} and Lemma~\ref{nu_and_hat_nu}, for any $i,j$ there
is $g \in G$ with $\widehat{\nu}(g\partial^i G \cap \partial^j G) > 0$. Hence by 
Lemma~\ref{typical_compare_lemma} there is an infinite geodesic ray in $G$ which is both
$E^i,\sigma^i$ typical and $E^j,\sigma^j$ typical. But this implies $E^i = E^j$ and $\sigma^i = \sigma^j$.
Since $i$ and $j$ were arbitrary, we have proved

\begin{lemma}\label{means_agree}
The means $E^i$ and standard deviations $\sigma^i$ as above are all equal to some $E,\sigma$.
\end{lemma}

Recall that the measure $\widehat{\nu}$ on $Y$ pushes forward to a measure on $Y_n$ by projection $Y \to Y_n$
and then to a measure $\widehat{\nu}$ on $G_n$ for all $n$. Note that $\widehat{\nu}$ on $G_n$
does not depend on the particular choice of a digraph $\Gamma$ parameterizing $L$, but just on $L$
itself. Lemma~\ref{means_agree} and
Theorem~\ref{clt_component} together imply our main result:

\begin{theorem}[Central Limit Theorem for bicombable functions]\label{clt_bicombable}
Let $\phi$ be a bicombable function on a word hyperbolic group with respect to some combing $L$. 
Let $\overline{\phi}_n$ be the
value of $\phi$ on a random word of length $n$ with respect to the $\widehat{\nu}$ measure. Then
there is convergence in the sense of distribution
$$\lim_{n \to \infty} n^{-1/2}(\overline{\phi}_n - nE) \to N(0,\sigma)$$
for some $\sigma \ge 0$ where $E$ denotes the mean of $d\phi$ on $\Gamma$ with respect
to the stationary measure $\mu$. In particular, $E$ and $\sigma$ are algebraic.
\end{theorem}
\begin{proof}
Fix some $g \in G_m$ which corresponds to an element $y \in Y_m$ such that the
last vertex $v_g = S^m y \in X_0 = \Gamma$ is in the support of $\mu$. For any
$n \ge m$, the shift map $S^m$ takes the subset $p^{-1}(y) \subset Y_n$ with
its $\widehat{\nu}$ measure isomorphically to the set of walks in $X$ of length $n-m$ starting
at $v_g$ with their $\mu$ measure, up to scaling the measures by a constant factor. 
Fixing $g$ and letting $n$ go to infinity, we
obtain the desired Gaussian distribution for the values of $\phi$ on $\cone(g) \cap G_n$ 
as $n \to \infty$ in the $\widehat{\nu}$ measure. 

For any $\epsilon$, there is $m$ such that the set of $g \in G_m$ 
with $v_g$ not in the support of $\mu$ has $\widehat{\nu}$ measure at most $\epsilon$. 
Since $\epsilon$ was arbitrary, we are done.
\end{proof}

The following corollary does not make reference to the measure $\widehat{\nu}$.
\begin{corollary}
Let $\phi$ be a bicombable function on a word-hyperbolic group $G$. Then there is a constant $E$ such that
for any $\epsilon > 0$ there is a $K$ and an $N$ so that if $G_n$ denotes the set of elements of length $n\ge N$, 
there is a subset $G_n'$ with
$|G_n'|/|G_n| \ge 1- \epsilon$, so that for all $g \in G_n'$, there is an inequality
$$|\phi(g) - nE| \le K \cdot \sqrt{n}$$
\end{corollary}

As a special case, let $S_1,S_2$ be two finite generating sets for $G$. Word length
in the $S_2$ metric is a bicombable function with respect to a combing $L_1$ for the $S_1$ generating
set. Hence:

\begin{corollary}\label{genset_length_compare}
Let $S_1$ and $S_2$ be finite generating sets for a word-hyperbolic group $G$. 
There is an algebraic number $\lambda_{1,2}$ such that for any
$\epsilon > 0$, there is a $K$ and an $N$ so that if $G_n$ denotes the set of elements of length $n \ge N$,
there is a subset $G_n'$ with $|G_n'|/|G_n| \ge 1-\epsilon$, so that for all $g \in G_n'$ there is an
inequality $$\bigl| |g|_{S_1} - \lambda_{1,2}|g|_{S_2} \bigr| \le K \cdot \sqrt{n}$$
\end{corollary}

\begin{remark}
A counting argument shows $\lambda_{1,2} \ge \log{\lambda_1}/\log{\lambda_2}$ where the
number of words of length $n$ in the $S_i$ metric is $\Theta(\lambda_i^n)$. 
But equality should only hold in very special
circumstances, and in fact the number $\lambda_{1,2}\log{\lambda_2}\log{\lambda_1}^{-1}$ is probably
very interesting in general. Such numbers have been studied in free groups in the special
case where $S_1$ and $S_2$ are free bases, by \cite{Kaimanovich_Kapovich_Schupp}.
\end{remark}

If $\phi$ is a quasimorphism, then $|\phi(g) + \phi(g^{-1})| \le \text{const.}$ so if $S$ is symmetric,
then necessarily $E$ as above is equal to $0$. Hence:

\begin{corollary}
Let $\phi$ be a bicombable quasimorphism on a word-hyperbolic group $G$. Let $\overline{\phi}_n$ be
the value of $\phi$ on a random word of length $n$ with respect to the $\widehat{\nu}$ measure. Then there
is convergence in the sense of distribution
$$\lim_{n \to \infty} n^{-1/2}\overline{\phi}_n \to N(0,\sigma)$$
for some $\sigma \ge 0$.
\end{corollary}

\section{H\"older quasimorphisms}\label{Holder_section}

The distribution of quasimorphisms on free groups has been studied by
Horsham-Sharp. We briefly describe their results. Let $F$ denote a free group
(on some finite generating set).

\begin{definition}
For any $g,a \in F$ and any $\psi:F \to \R$ define
$$\Delta_a \psi(g) = \psi(g) - \psi(ag)$$
For $x,y \in F$ let $(x|y)$ denote the Gromov product
$$(x|y):=(|x| + |y| - |x^{-1}y|)/2$$
A quasimorphism $\psi$ on $F$ is {\em H\"older} if for any $a \in F$ there are
constants $C,c>0$ such that for any $x,y \in F$ there is an inequality
$$|\Delta_a \psi(x) - \Delta_a \psi(y)| \le Ce^{-c(x|y)}$$
\end{definition}

Note that the constants $C,c$ depend on $a$ but not on $x$ or $y$.
The main theorem of \cite{Horsham_Sharp}, which also appeared in Matthew Horsham's PhD thesis, is
the following:

\begin{theorem}[Horsham-Sharp \cite{Horsham_Sharp}]
Let $\psi$ be a H\"older quasimorphism on a free group. If $\overline{\psi}_n$ denotes the
value of $\psi$ on a random word in $F$ of length $n$ (with respect to a standard generating set)
then there is convergence in the sense of distributions
$$n^{-1/2}\overline{\psi}_n \to N(0,\sigma)$$
for some $\sigma$.
\end{theorem}

These results can also be generalized to fundamental groups of closed (hyperbolic) surfaces.
Big (Brooks) counting quasimorphisms on free groups are H\"older, since $\Delta_a\psi(x) = \Delta_a\psi(y)$
whenever $(x|y)$ is bigger than $|a|$. But small counting quasimorphisms (i.e. counting quasimorphisms where
copies must be disjoint) are not, as the following
example shows:

\begin{example}
Let $\phi$ be the small counting quasimorphism for the word $abab$. Then
$$\phi(\underbrace{babab\cdots ab}_{4n+1}) = n, \quad \phi(\underbrace{ababab\cdots ab}_{4n+2}) = n$$
but
$$\phi(\underbrace{babab\cdots ab}_{4n+3}) = n, \quad \phi(\underbrace{ababab\cdots ab}_{4n+4}) = n+1$$
\end{example}

Of course, since small counting quasimorphisms are bicombable, 
Theorem~\ref{clt_bicombable} applies. Note that the measure $\widehat{\nu}$ agrees with the
uniform measure in a free group with the standard generating set.

It seems plausible that Horsham-Sharp's methods apply to arbitrary H\"older 
quasimorphisms on hyperbolic groups, though we have not pursued this.

\section{Acknowledgments}

We would like to thank Ian Agol, Ilya Kapovich, Shigenori Matsumoto, Curt McMullen, Mark Pollicott,
Richard Sharp, Amie Wilkinson and the anonymous referees for useful comments and corrections.
This research was largely carried out while the first author was at the
Tokyo Institute of Technology, hosted by Sadayoshi Kojima. He thanks TIT
and Professor Kojima for their hospitality.  The second author
would like to thank Caltech for their hospitality in hosting him while
some of this paper was written up. The paper \cite{Picaud} was
a catalyst for discussions which led to the content of
this paper, and we are pleased to acknowledge it. Danny Calegari was supported
by NSF grant DMS 0707130.

\end{document}